\documentclass[a4paper,11pt]{amsart}

\usepackage{amsfonts,amssymb}
\usepackage{textcomp, color}
\usepackage{enumerate}

\theoremstyle{plain}

\newtheorem*{thmA}{Theorem A}
\newtheorem*{thmB}{Theorem B}

\newtheorem{thm}{Theorem}[section]
\newtheorem{lem}[thm]{Lemma}
\newtheorem{pro}[thm]{Proposition}

\theoremstyle{definition}

\newtheorem{rmk}[thm]{Remark}

\numberwithin{equation}{section}

\newcommand{\F}{\mathbb{F}}
\newcommand{\Z}{\mathbb{Z}}

\newcommand{\N}{\mathbb{N}}

\newcommand{\TT}{\mathcal{T}}

\DeclareMathOperator{\Aut}{Aut}

\DeclareMathOperator{\Cl}{Cl}

\DeclareMathOperator{\st}{st}

\begin{document}

\title[GGS-groups as a source for Beauville surfaces]{Grigorchuk-Gupta-Sidki groups as a source for Beauville surfaces}

\author[\c{S}.\ G\"ul]{\c{S}\"ukran G\"ul}
\address{Department of Mathematics\\ University of the Basque Country UPV/EHU\\
48080 Bilbao, Spain}
\email{sukran.gul@ehu.eus}

\author[J.\ Uria-Albizuri]{Jone Uria-Albizuri}
\address{Basque Center for Applied Mathematics BCAM\\
	48009 Bilbao, Spain}
\email{juria@bcamath.org}

\keywords{Finite $p$-groups; Beauville $p$-groups; Beauville surfaces; authomorphisms of trees; GGS-groups.\vspace{3pt}}

\thanks{Authors acknowledge financial support from the Spanish Government, grants MTM2014-53810-C2-2-P and MTM2017-86802-P, and from the Basque Government, grant IT974-16.}

\begin{abstract}
If $G$ is a Grigorchuk-Gupta-Sidki group defined over a $p$-adic tree, where $p$ is an odd prime, we study the existence of Beauville surfaces associated to the quotients of $G$ by its level stabilizers $\st_G(n)$. We prove that if $G$ is periodic then the quotients $G/\st_G(n)$ are Beauville groups for every $n\geq 2$ if $p\geq 5$ and $n\geq 3$ if $p=3$. On the other hand, if $G$ is non-periodic, then none of the quotients $G/\st_G(n)$ are Beauville groups.
 \end{abstract}

\maketitle

\section{Introduction}

Groups acting on regular rooted trees have been widely studied since the 1980's, when the first Grigorchuk group was defined by Rostilav Grigorchuk \cite{gri}. This group was designed to be a counterexample to the General Burnside Problem, so that it is a finitely generated, periodic and infinite group. More importantly, it was the first example of a group having intermediate word growth \cite{gri2}, answering the Milnor Problem \cite{mil}. Later on many different examples and generalizations came into the literature. The interest on this kind of groups resides in their \textit{odd} properties, and because they have been useful to answer unsolved questions. 

Some of the examples that came up were the Gupta-Sidki groups \cite{GS} and the second Grigorchuk group \cite{gri}. The Grigorchuk-Gupta-Sidki groups (GGS-groups for short) are a family of groups generalizing them. These groups act on the regular $p$-adic rooted tree where $p$ is an odd prime. More concretely, each of them is generated by two automorphisms: a rooted automorphism $a$ permuting the vertices hanging from the root according to the permutation $(1\,2\,\dots\,p)$, and a recursively defined automorphism $b$ which is defined according to a given vector $\bold{e}=(e_1,\dots,e_{p-1})\in\F_p^{p-1}$. Vovkivsky \cite{vov} showed that $G$ is always infinite and it is a periodic group if and only if $\Sigma_{i=1}^{p-1}e_i=0$, so some of them are also counterexamples for the General Burnside Problem.

A group $G$ acting faithfully on a regular rooted tree $\TT$ is always residually finite, and a way to analyze the structure of a residually finite group is by looking at its finite quotients. For these groups there is a very natural family of normal subgroups of finite index, which are the level stabilizers $\st_G(n)$ for each $n\in\N$. The quotient $G/\st_G(n)$ can be naturally seen as a subgroup of the group of automorphisms of the subtree $\TT_n$ consisting of the first $n$ levels, since the kernel of the action of $G$ on $\TT_n$ is $\st_G(n)$.
For the GGS-groups these quotients have been well studied. For instance, in \cite{FZ} Fern\'andez-Alcober and Zugadi-Reizabal have given the sizes of these quotients, and  the profinite completion of each group has been compared to the completion with respect to the level stabilizers by Fern\'andez-Alcober, Garrido and Uria-Albizuri  in \cite{FGU}. The aim of this paper is to determine whether these quotients are Beauville groups or not.

A finite group $G$ is called a \emph{Beauville group} if it is a $2$-generator group and there exists a pair of generating sets $\{x_1,y_1\}$ and $\{x_2,y_2\}$ of $G$ such that 
$\Sigma(x_1,y_1) \cap \Sigma(x_2,y_2)=1$, where
\[
\Sigma(x_i,y_i)
=
\bigcup_{g\in G} \,
\Big( \langle x_i \rangle^g \cup \langle y_i \rangle^g \cup \langle x_iy_i \rangle^g \Big),
\]
for $i=1, 2$. Then we say that $ \{x_1,y_1\}$ and $\{x_2,y_2\}$ form a \emph{Beauville structure\/} for $G$.

Every Beauville group gives rise to a complex surface of general type which is known as a \emph{Beauville surface}. Roughly speaking, a Beauville surface is a compact complex surface defined by taking a pair of complex curves $C_1$ and $C_2$ and letting a finite group $G$, which is called a Beauville group, act freely on their product to define the surface as the quotient $(C_1 \times C_2)/G$.

Beauville groups have been intensely studied in recent times; see surveys \cite{fai, jon}. 
For example, the abelian Beauville groups were classified by Catanese [4]: a finite abelian group $G$ is a Beauville group if and only if $G\cong C_n \times C_n$ for $n> 1$ with $\gcd(n, 6) = 1$. After abelian groups, the most natural class of finite groups to consider are nilpotent groups. The study of nilpotent Beauville groups is reduced to that of Beauville $p$-groups.

If $p$ is a prime, the state of the art on Beauville $p$-groups can be found in the survey papers \cite{BBF}, \cite{bos} and \cite{fai2}. The smallest non-abelian Beauville $p$-groups were determined by Barker, Boston and Fairbairn in \cite{BBF}. On the other hand, in \cite{FG}, Fern\'andez-Alcober and G\"ul extended Catanese's criterion in the case of $p$-groups from abelian groups to  finite $p$-groups having a `nice power structure', including in particular $p$-groups of class $<p$.

Also in \cite{BBF}, it was shown that there are non-abelian Beauville $p$-groups of order $p^n$ for every $p\geq5$ and every $n\geq3$. The first explicit infinite family of Beauville $2$-groups was constructed in \cite{BBPV}. Recently in \cite{SV}, Stix and Vdovina constructed an infinite series of Beauville $p$-groups, for every prime $p$, by considering quotients of ordinary triangle groups. In particular this gives examples of non-abelian Beauville $p$-groups of arbitrarily large order. On the other hand, in \cite{gul}, G\"ul showed that quotients by the terms of the lower $p$-central series in either the free group of rank $2$ or in the free product of two cyclic groups of order $p$ are Beauville groups. In \cite{FG}, quotients of the Nottingham group over $\F_p$ have been studied in order to construct more infinite families of Beauville $p$-groups, for an odd prime $p$.

Note that if $G$ is a GGS-group, then the quotients of $G$ by its level stabilizers $\st_G(n)$ are finite $p$-groups generated by two elements of order $p$ but whose exponent can be arbitrarily high. As a consequence, they do not fit into the family of finite $p$-groups having a `nice power structure'. For this reason, they are natural candidates to search for Beauville $p$-groups of a very different type from the ones in \cite{FG}. It turns out that the property of being Beauville for these quotients depends on whether $G$ is periodic or not.

The main results of this paper are as follows.

\begin{thmA}
Let $G$ be a periodic GGS-group over the $p$-adic tree. Then the quotient $G/\st_G(n)$ is a Beauville group if $p \geq5$ and $n\geq2$,  or $p=3$ and $n\geq 3$.
\end{thmA}

\begin{thmB}
Let $G$ be a non-periodic GGS-group over the $p$-adic tree. Then the quotient $G/\st_G(n)$ is not a Beauville group for any $n \geq 1$.
\end{thmB}

Theorem A shows that a periodic GGS-group is a source for the construction of an infinite series of Beauville $p$-groups. This gives  yet another reason why GGS-groups constitute an important family in group theory.

\vspace{10pt}

\noindent
\textit{Notation.\/}
 If $G$ is a group, then we denote by $\Cl_G(x)$ the conjugacy class of the element $x\in G$. Also, if $p$  is a prime, then the exponent of a $p$-group $G$, denoted by $\exp G$, is the maximum of the orders of all elements of $G$.

\section{Definitions and Preliminaries}

In this section, we will establish a few properties of GGS-groups that will help us prove the main theorems of this paper. Before proceeding, we recall some facts about automorphisms of rooted trees.

If $d\geq 2$ is an integer and $X = \{1, . . . , d\}$, the \emph{$d$-adic tree $\TT$} is the rooted tree whose set of vertices is the free monoid $X^*$, where the root corresponds to the empty word $\emptyset$, and a word $u$ is a descendant of $v$ if $u = vx$ for some $x \in X$. The set $L_n$ of all vertices of length $n$ is called the \emph{$n$th level of $\TT$}, for every integer $n\geq 0$. If we consider only words of length $\leq n$, then we have a finite tree $\TT_n$, which we refer to as \emph{the tree $\TT$ truncated at level $n$}.

An \emph{automorphism of $\TT$} is a bijection of the vertices that preserves incidence. The group of automorphisms of $\TT$ is denoted by $\Aut \TT$. The subgroup $\st(n)$ of $\Aut \TT$ consisting of the automorphisms that fix pointwise $L_n$ is called the \emph{nth level stabilizer}. More generally, if $G \leq \Aut \TT$, we define $\st_G(n)=\st(n)\cap G$.

If an automorphism $g$ fixes a vertex $u$, then the restriction of $g$ to the subtree hanging from $u$ induces an automorphism $g_u$ of $\TT$. In particular, if $g\in \st(1)$ then $g_i$ is defined for every 
$i = 1,...,d$, and we have an isomorphism
\begin{align*} 
\psi \colon \st(1) & \longrightarrow \Aut \TT \times  \overset{d}{\ldots} \times  \Aut \TT\\ 
g & \longmapsto (g_1, \dots, g_d).
\end{align*}

An important automorphism of $\TT$ is the automorphism that permutes the $d$ subtrees hanging from the root rigidly according to the permutation $(1  2 . . .  d)$. This is called a \emph{rooted automorphism} and will be denoted by the letter $a$. Since $a$ has order $d$, it makes sense to write $a^k$ for $k \in \Z/d\Z$. Now, given a non-zero vector $\bold{e} = (e_1, . . . , e_{d-1})\in (\Z/d\Z)^{d-1}$, we can
define recursively $b \in \st(1)$ via
\[
\psi(b)= (a^{e_1}, \dots, a^{e_{d-1}}, b).
\]
Then the subgroup $G=\langle a, b\rangle$ of $\Aut \TT$ is called the \emph{GGS-group} corresponding to the \emph{defining vector} $\bold{e}$.

If $d=2$ then there is only one GGS-group, which is isomorphic to the infinite dihedral group $D_{\infty}$. In case $d=3$ there are only three essentially different defining vectors $\bold{e}$, and they are $(1, 0)$, $(1,1)$ and $(1,2)$. The corresponding groups are called  the Fabrykowski-Gupta group, the Bartholdi-Grigorchuk group and the Gupta-Sidki group, respectively. If $d=4$ then one example of GGS-group is the second Grigorchuk group, where the defining vector is $(1,0,1)$.

In the remainder of this paper, we will asume that $d=p$ for an odd prime $p$. Let $G=\langle a, b\rangle$ be a GGS-group with defining vector $\bold{e}=(e_1, \dots, e_{p-1})$. 
Observe that both $a$ and $b$ are of order $p$. For every integer $i$, we write $b_i=b^{a^i}$. Notice that  $b_i = b_j$  if $i \equiv j \pmod{p}$.
The images of the elements $b_i$ under the map $\psi$ can be described as:
\begin{equation}
\begin{aligned}
&\psi(b_0)= (a^{e_1}, \dots, a^{e_{p-1}}, b), \\
& \psi(b_1)= (b, a^{e_1}, \dots, a^{e_{p-1}}),\\
&\vdots \\
& \psi(b_{p-1})= (a^{e_2}, a^{e_3}, \dots,b, a^{e_1}).
\end{aligned}
\end{equation}

Here we collect some basic results regarding GGS-groups.

\begin{pro}\cite[Theorem 2.1]{FZ} 
\label{description of G}
If $G=\langle a, b\rangle$ is a GGS-group, then 
\begin{enumerate}
\item 
$\st_G(1)=\langle b \rangle^{G}= \langle b_0, \dots, b_{p-1} \rangle$, and  $G= \langle a \rangle \ltimes \st_G(1)$.
\item 
$\st_G(2) \leq G' \leq \st_G(1)$.
\item 
$|G:G'|=p^2$ and $|G: \gamma_{3}(G)|=p^3$.
\end{enumerate}
\end{pro}

\begin{rmk}
\label{image of st_G(k)}
For all $k\geq 1$ we have $ \psi(\st_G(k)) \subseteq  \st_G(k-1)\times  \overset{p}{\ldots} \times  \st_G(k-1)$.
\end{rmk}

\begin{rmk}
\label{image of st_{G_n}(k)}
For every positive integer $n$, we can define an isomorphism $\psi_n$ from the stabilizer
of the first level in $\Aut \TT_n$ to the direct product 
$\Aut \TT_{n-1} \times  \overset{p}{\ldots} \times \Aut \TT_{n-1}$, in the same way as $\psi$ is defined.
Since $G_n=G/\st_G(n)$ can be seen as a subgroup of $\Aut \TT_n$, we can
consider the restriction of $\psi_n$ to $\st_{G_n}(1)$. Then by the previous remark, we have
\[
\psi_n(\st_{G_n}(k)) \subseteq  \st_{G_{n-1}}(k-1)\times  \overset{p}{\ldots} \times \st_{G_{n-1}}(k-1).
\]
For simplicity, throughout this paper we are not going to use bar notation for the elements in the quotient groups $G/\st_G(n)$. However, the subscript $n$ at the map $\psi_n$ means that we are working in the quotient group $G/\st_G(n)$.
\end{rmk}

If $\bold{e}=(e_1, \dots, e_{p-1})$ is the defining vector of a GGS-group, then we write $C(\bold{e}, 0)$ for the circulant matrix $C(e_1,...,e_{p-1}, 0)$ over $\F_p$.

\begin{pro}\cite[Theorem 2.4]{FZ} 
\label{properties of G2}
Let $G$ be a GGS-group with defining vector $\bold{e}$, and put $C=C(\bold{e}, 0)$. Then 
\begin{enumerate}
\item 
The dimension of $\st_{G_2}(1)$ coincides with the rank $t$ of $C$.
\item 
$G_2$ is a $p$-group of maximal class of order $p^{t+1}$.
\end{enumerate}
\end{pro}

\begin{lem}\cite[Lemma 2.7]{FZ} 
\label{rank of C}
Let $C=C(e_1, \dots, e_{p-1}, 0)$  be a circulant matrix over $\F_p$. Then 
\begin{enumerate}
\item 
The rank of $C$ is $p-m$, where $m$ is the multiplicity of $1$ as a root of the polynomial 
$f(X)=e_1+e_2X+\dots + e_{p-1}X^{p-2}$.
\item 
The rank of $C$ is strictly less than $p$ if and only if $\Sigma_{i=1}^{p-1}e_i=0$.
\end{enumerate}
\end{lem}

\begin{pro}\cite[Theorems 2.13, 2.14]{FZ} 
\label{comm of st_G(1)}
Let $G$ be a GGS-group. Then 
\begin{enumerate}
\item 
$\st_{G}(1)^{'} \leq \st_G(2)$.
\item 
$|G: \st_G(1)^{'}|=p^{p+1}$.
\end{enumerate}
\end{pro}

We say that the defining vector $\bold{e}=(e_1, \dots, e_{p-1})$ of a GGS-group is \emph{symmetric} if $e_i=e_{p-i}$ for all $i=1,\dots, p-1$.

\begin{pro}\cite[Lemma 3.4, Theorem 3.7]{FZ}
\label{non-symmetric}
Let $G$ be a GGS-group with non-symmetric defining vector $\bold{e}$. Then 
\begin{enumerate}
\item 
$\psi(st_G(1)')=G' \times  \overset{p}{\ldots} \times G'$.
\item 
 If $G_n=G/\st_G(n)$ then for every $n\geq 2$
 \[
 \log_p |G_n|=tp^{n-2}+1,
 \]
 where $t$ is the rank of $C(\bold{e},0)$.
\end{enumerate}
\end{pro}

\section{Quotients of periodic GGS-groups}

Let $G$ be a periodic GGS-group with defining vector $\bold{e}=(e_1, \dots, e_{p-1})$, where $p$ is an odd prime. In this section, we determine whether the quotients of $G$ by its level stabilizers $\st_G(n)$ are Beauville groups.

We first deal with the quotient group $G_2=G/\st_G(2)$.

\begin{thm}
\label{G_2 periodic}
Let $G$ be a periodic GGS-group. Then the quotient $G/\st_G(2)$ is a Beauville group if and only if $p\geq 5$.
\end{thm}

\begin{proof}
By Proposition \ref{properties of G2}(ii), we know that $G_2$ is a group of maximal class of order $p^{t+1}$, where $t$ is the rank of $C(\bold{e}, 0)$. Then Lemma \ref{rank of C}(ii) implies that $|G_2|\leq p^p$, and hence $G_2$ is regular by Lemma 3.13 in \cite{suz2}.  Since $G_2$ is generated by two elements of order $p$, this, together with being regular, implies that $\exp G_2=p$, by Corollary 2.11 in \cite{fer}. Then according to Corollary 2.10 in \cite{FG}, $G_2$ is a Beauville group if and only if $p\geq 5$.
\end{proof}

We next deal with the quotient $G_3=G/\st_G(3)$. To this purpose, first we need to calculate the orders of $a^{-1}b$ and $ab^i$ for $1\leq i \leq p-1$.

\begin{lem}
\label{orders of ab^i}
Let $G$ be a periodic GGS-group. Then the orders of  $a^{-1}b$ and $ab^i$ for $1\leq i \leq p-1$ are $p^2$ in both $G$ and $G_3$.
\end{lem}

\begin{proof}
Let $1\leq i \leq p-1$. Then we have
\[
(ab^i)^p= a^p(b^i)^{a^{p-1}}\dots (b^i)^ab^i= b_{p-1}^i\dots b_1^i b_0^i,
\]
and hence 
\[
\psi((ab^i)^p)=\big( a^{i(e_2+\dots + e_{p-1})}b^i a^{ie_1}, a^{i(e_3+\dots + e_{p-1})}b^i a^{i(e_1+e_2)}, \dots, a^{i(e_1+\dots + e_{p-1})}b^i  \big).
\]
Then being $\Sigma_{i=1}^{p-1} e_i=0$ implies that
\begin{equation}
\label{pth power of ab^i}
\psi((ab^i)^p)=\big( (b^i)^{a^{ie_1}}, (b^i)^{a^{i(e_1+e_2)}}, \dots, b^i, b^i  \big).
\end{equation}
Since each component in equation (\ref{pth power of ab^i}) is of order $p$, the order of $ab^i$ is $p^2$ for every $1\leq i \leq p-1$. Observe that $(ab^i)^p \notin \st_G(3)$, and hence also in $G_3$ the order of $ab^i$ is $p^2$.

On the other hand, since  $(a^{-1}b)^{-1}=(ab^{-1})^b=(ab^{p-1})^b$, the order of $a^{-1}b$ is also $p^2$ in  both $G$ and $G_3$.
\end{proof}

The next lemma shows the relation between the conjugates of $b$ by powers of $a$ modulo $\st_G(2)$.

\begin{lem}
\label{conjugates of b}
Let $G$ be a GGS-group. If 
\[
b^{a^i}\equiv b^{a^j} \pmod{\st_G(2)}
\]
for $0\leq i, j \leq p-1$, then $i=j$.
\end{lem}

\begin{proof}
Suppose, on the contrary, that $i\neq j$. Then $G=\langle a^{j-i},  b \rangle$. The congruence $b^{a^i}\equiv b^{a^j} \pmod{\st_G(2)}$ implies that
$[b, a^{j-i}] \in \st_G(2)$, and so $G/\st_G(2)$ is abelian. However, by Lemma 
\ref{rank of C}(i), the rank of the circulant matrix $C(\bold{e}, 0)$ is at least $2$. Thus, $G/\st_G(2)$ is  a $p$-group of maximal class of order at least $p^3$, and hence it is not abelian.
\end{proof}

The following proposition is the key result for determining the existence of a Beauville structure for $G_3$.

\begin{pro}
\label{key pro}
Let $G$ be a periodic GGS-group. If
\[
\langle (ab^i)^p\rangle= \langle (ab^j)^p\rangle^g
\]
for $1\leq i, j \leq p-1$ and $g \in G_3$, then $i=j$. 
\end{pro}

\begin{proof}
For simplicity we write $w_i=(ab^i)^p$ for every $1\leq i \leq p-1$.
Then the equality $\langle w_i\rangle= \langle w_j\rangle^g$ implies that
\begin{equation}
\label{coincide}
w_i^k=w_j^g
\end{equation}
for some $1\leq k \leq p-1$.
By (\ref{pth power of ab^i}),  we have 
\begin{equation}
\label{pth power of ab^i in G_3}
\psi_3(w_i)=\big( (b^i)^{a^{ie_1}}, (b^i)^{a^{i(e_1+e_2)}}, \dots, b^i, b^i  \big).
\end{equation}
By Remark \ref{image of st_{G_n}(k)}, each component in the equation (\ref{pth power of ab^i in G_3}) is an element of $G_2$.

Since $\st_{G_2}(1)$ is abelian, the components of $\psi_3(w_j^g)$ are of the form $(b^j)^{a^m}$, with $m \in \{0, \dots, p-1\}$. Also one of the components of $\psi_3(w_i^k)$ is $b^{ik}$. Then by equality (\ref{coincide}), we have  $(b^j)^{a^m}=b^{ik}$ in $G_2$ for some $m$. Thus, $b^{ik-j}=[b^j, a^m] \in G_2'$, and this is true only if $ik-j\equiv 0 \pmod{p}$. Therefore, $k \equiv i^{-1}j \pmod{p}$, and hence we have $w_i^{i^{-1}j}=w_j^g$. If we write $x_i=w_i^{i^{-1}}$ for every $1\leq i \leq p-1$, then we have 
\[
x_i=x_j^g.
\]
Note that
\[
\psi_3(x_i)=\big( b^{a^{ie_1}}, b^{a^{i(e_1+e_2)}}, \dots, b, b\big)
\]
for every $1\leq i \leq p-1$.

Let $g=a^sh_s$ for some $0 \leq s \leq p-1$ and $h_s \in \st_{G_3}(1)$. Observe that
\[
\psi_3(x_j^{a^s})=\big(b^{a^{j(e_1+\dots+e_{p-(s-1)})}}, \dots, b, b, \dots, b^{a^{j(e_1+\dots+ e_{p-s})}}  \big),
\]
where the first $b$ appears at the $(s-1)$st component. Then the equality $x_i=x_j^g$ implies that
\begin{align*}
\psi_3(h_s)= \big( a^{ie_1-j(e_1+\dots+e_{p-(s-1)})}u_1, a^{i(e_1+e_2)- j(e_1+ \dots+ e_{p-(s-2)})}&u_2 ,  \dots, \\
 &a^{-j(e_1+\dots+e_{p-s})}u_p  \big),
\end{align*}
where $u_{\ell} \in \st_{G_2}(1)$ for all $1\leq \ell \leq p$. We next define recursively elements $h_{i-1}=h_ib_{i-1}^{-j}$ for $i=s, \dots, 1$. Now since $G$ is periodic, we have $e_1+\dots + e_{p-1}=0$ and consequently
\[
\psi_3(h_0)=\big( a^{(i-j)e_1} v_1, a^{(i-j)(e_1+e_2)}v_2, \dots, v_{p-1}, v_p \big),
\]
with $v_{\ell} \in \st_{G_2}(1)$ for all $1\leq \ell \leq p$. 

Notice that we have
\[
\psi_3(h_0^{a}) \equiv \psi_3(h_0b^{j-i}) \pmod{\st_{G_2}(1) \times \overset{p}{\ldots} \times \st_{G_2}(1)}.
\]
Hence $\psi_3(b^{i-j}[h_0,a]) \in \big(\st_{G_2}(1) \times \overset{p}{\ldots} \times \st_{G_2}(1) \big) \cap \psi_3(\st_{G_3}(1))=\psi_3(\st_{G_3}(2))$. Thus, $b^{i-j}[h_0,a] \in \st_{G_3}(2)$, and hence  $b^{i-j} \in \st_{G_3}(2)$. This implies that $i=j$.
\end{proof}

 We are now ready to  prove that $G_3$ is a Beauville group. We deal separately with the cases $p\geq 5$ and $p=3$.

\begin{thm}
\label{ G_3 periodic}
Let $G$ be a periodic GGS-group over the $p$-adic tree. If $p\geq 5$ then the quotient $G/\st_G(3)$ is a Beauville group.
\end{thm}

\begin{proof}
Note that $\{a^{-2}, ab \}$ and $\{ab^2, b\}$ are both systems of generators of $G_3=G/\st_G(3)$. We claim that they yield a Beauville structure for $G_3$. If $X=\{ a^{-2}, ab, a^{-1}b \} $ and $Y=\{ab^2, b, ab^3 \}$, we have to see that
\begin{equation}
\label{check}
\langle x \rangle \cap \langle y^g \rangle=1
\end{equation}
for all $x\in X$, $y\in Y$, and $g\in G_3$.
Observe that $\langle x\Phi(G_3) \rangle$ and $\langle y\Phi(G_3) \rangle$ have trivial intersection for every $x\in X$ and $y\in Y$, since $a$ and $b$ are linearly independent modulo $\Phi(G_3)$. Thus, $x$ and $y^g$ lie in different maximal subgroups of $G_3$ in every case.

Assume first that $x=a^{-2}$, which is an element of order $p$.
If (\ref{check}) does not hold, then $\langle x \rangle \subseteq \langle y^g \rangle$, and consequently $\langle x\Phi(G_3) \rangle=\langle y\Phi(G_3) \rangle$, which is a contradiction. The same argument holds if $y=b$ since it is also of order $p$.

The remaining elements in $X$ and $Y$ are of order $p^2$, by Lemma \ref{orders of ab^i}. Thus, in order to prove our claim, we need to show that 
\begin{equation}
\label{pth power check}
\langle x^p \rangle \neq \langle y^p\rangle^g
\end{equation}
for all $g \in G_3$, and for those $x \in X$ and $y\in Y$. If $x=ab$ and $y=ab^2$ or $ab^3$, then we apply Proposition \ref{key pro} to conclude that (\ref{pth power check}) holds.

It remains to deal with $x=a^{-1}b$ and $y=ab^2$ or $y=ab^3$. Since $(a^{-1}b)^{-1}=(ab^{p-1})^b$, if (\ref{pth power check}) does not hold, then 
$\langle (ab^{p-1})^p\rangle^b=\langle y^p\rangle^g$, that is, 
\[
\langle (ab^{p-1})^p\rangle=\langle y^p\rangle^{gb^{-1}},
\]
for some $g \in G_3$. This contradicts with Proposition \ref{key pro}, and hence (\ref{pth power check}) holds. This completes the proof.
\end{proof}


Now we assume that $p=3$.

Since proportional nonzero vectors define the same GGS group, if $G$ is periodic and $p=3$, then the defining vector of $G$ can only be $\bold{e}=(1,-1)$. So $G$ is the Gupta-Sidki $3$-group. In this case, observe that the rank of $C(\bold{e}, 0)$ is $2$. Then Proposition \ref{properties of G2}(ii), together with Proposition \ref{description of G}(iii), implies that $\st_G(2)=\gamma_3(G)$, which is of index $3^3$ in $G$. By Proposition \ref{non-symmetric}(ii), the order of $G_3$ is $3^7$. Since the smallest Beauville $3$-group is of order $3^5$ \cite{BBF}, we can ask whether $G_3$ is Beauville or not.


\begin{lem}
\label{center}
 Let $G=\langle a, b\rangle$ be the Gupta-Sidki $3$-group. Then $Z(G_3)$ is a subgroup of $\st_{G_3}(1)'$ of order $3$. More precisely, if $Z(G_3)=\langle z\rangle$ then 
 \begin{equation}
 \label{image of center}
 \psi_3(z)=([a,b], [a,b], [a,b]).
 \end{equation}
\end{lem}

\begin{proof}
First of all, we observe that  $Z(G_3) \leq \st_{G_3}(1)$. Otherwise, if there were $z\in Z(G_3)$ such that 
$z\notin  \st_{G_3}(1)$, then $G_3=\langle z, b \rangle$ would be an abelian group, which is not true.

Since for every $n\geq 1$, $\st_G(n)$  is a subdirect product of $3^n$ copies of $G$, this, together with 
\[
\psi_3(\st_{G_3}(1)) \subseteq G_2\times G_2 \times G_2,
\]
implies that
\begin{equation}
\label{image of Z(G_3)}
\psi_3(Z(G_3)) \subseteq Z(G_2)\times Z(G_2) \times Z(G_2)=G_2'\times G_2' \times G_2'.
\end{equation}
On the other hand, by Proposition \ref{non-symmetric} (i), we have
\begin{equation}
\label{image of st(1)'}
\psi_3(\st_{G_3}(1)')=G_2'\times G_2' \times G_2'.
\end{equation}
Thus, by (\ref{image of Z(G_3)}) and (\ref{image of st(1)'}), we have $Z(G_3) \leq \st_{G_3}(1)'$. Let $z \in Z(G_3)$. Since $z=z^a$, this yields that $\psi_3(z)=(c,c,c)$ for  some $c \in G_2'=\langle [a,b]\rangle$. Hence $|Z(G_3)|=3$ and $Z(G_3)=\langle z\rangle$, where $\psi_3(z)=([a,b], [a,b], [a,b])$.
\end{proof}

\begin{lem}
\label{comms of b}
Let $G=\langle a, b\rangle$ be the Gupta-Sidki $3$-group. Then
\[
Z(G_3) \cap \{ [b,g] \mid g \in G_3\}=1.
\]
\end{lem}

\begin{proof}
Suppose that $1\neq [b,g] \in Z(G_3)$. Then since $Z(G_3) \leq \st_{G_3}(1)'$, $b$ and $g$ commute modulo $\st_{G_3}(1)'$.  We will first show that $g\in \st_{G_3}(1)$.

Since 
$|G_3/\st_{G_3}(1)': \gamma_2(G_3/\st_{G_3}(1)') |=3^2$ and 
$\st_{G_3}(1)/\st_{G_3}(1)'$ is an abelian maximal subgroup of $G_3/\st_{G_3}(1)'$, the quotient group $G_3/\st_{G_3}(1)'$ is a $p$-group of maximal class with an abelian maximal subgroup. Then being $b \in \st_{G_3}(1)$ yields that $g\in \st_{G_3}(1)$.

If $g\in \st_{G_3}(1)'$ then by (\ref{image of st(1)'}), $\psi_3(g) \in Z(G_2\times G_2 \times G_2)$. This implies that $\psi_3([b,g])=(1,1,1)$, which is a contradiction. Hence $g \in \st_{G_3}(1) \smallsetminus  \st_{G_3}(1)'$.

Write $g=b_0^{i_0}b_1^{i_1}b_2^{i_2}c$ for some $c \in \st_{G_3}(1)'$. Then 
\begin{align*}
\psi_3([b,g])&=\big([a, a^{i_0}b^{i_1}a^{2i_2}], [a^2, a^{2i_0+i_1}b^{i_2}], [b, b^{i_0}a^{2i_1+i_2}] \big) \\
&=\big([a, a^{i_0}b^{i_1}a^{2i_2}], \ b_2^{-i_2}b_0^{i_2}, \ b_0^{-1}b_{2i_1+i_2} \big).
\end{align*}
Since $[b,g] \in Z(G_3)$, it follows that $\psi_3([b,g])=\big( [a,b], [a,b], [a,b] \big)^{\pm1}$. Note that in $G_2$, we have $b_0b_1b_2=1$, and thus $[a,b]=b_1^{-1}b_0=b_2b_0^{-1}$. Then by the second and third components, we get $i_1=0$. Then the first component will be $1$, which is a contradiction.
\end{proof}

\begin{lem}
\label{comms of a}
Let $G=\langle a, b\rangle$ be the Gupta-Sidki $3$-group. Then the element $v \in \st_{G_3}(1)'$ such that $\psi_3(v)=([a,b], 1, 1)$ is not in the set
$\{[a,g] \mid g \in G_3 \}$.
\end{lem}

\begin{proof}
Since $\psi_3(\st_{G_3}(1)')=G_2'\times G_2' \times G_2'$, such an element $v$ exists in $ \st_{G_3}(1)'$. Suppose that $v=[a,g]$ for some $g \in G_3$. If we write $g=a^ih$ for some $h\in \st_{G_3}(1)$ then $v=[a,h]$. Write $\psi_3(h)=(h_1, h_2, h_3)$. Then
\[
\psi_3((h^{-1})^ah)=(h_3^{-1}h_1, h_1^{-1}h_2, h_2^{-1}h_3)=([a,b], 1, 1).
\]
This implies that $h_1=h_2=h_3$ in $G_2$. Then $[a,b]=h_3^{-1}h_1=1$ in $G_2$, which is a contradiction. Thus, $v \in  \st_{G_3}(1)' \smallsetminus \{[a,g] \mid g \in G_3 \}$.
\end{proof}

In order to deal with the prime $3$, we also need the following lemma.

\begin{lem}\textup{\cite[Lemma~3.8]{FG} }
\label{intersection}
Let $G$ be a finite $p$-group and let $x \in G \smallsetminus \Phi(G)$ be an element of order $p$.
If $t \in \Phi(G)\smallsetminus \{[x,g] \mid g \in G \}$ then 
\[
\Big(\bigcup_{g\in G} {\langle x\rangle}^g  \Big)
\bigcap
\Big(\bigcup_{g\in G} {\langle xt\rangle}^g \Big)= 1.
\]
\end{lem}

\begin{thm}
\label{ G_3 periodic, p=3}
Let $G$ be the Gupta-Sidki $3$-group. Then the quotient $G/\st_G(3)$ is a Beauville group.
\end{thm}

\begin{proof}
 Let $1 \neq u \in Z(G_3)$ and let $v \in \st_{G_3}(1)'$ be such that $\psi_3(v)=([a,b], 1, 1)$.
 We claim that $\{a, b\}$ and $\{av, b^2u \}$ form a Beauville structure for $G_3$. Let
 $X=\{a, b, ab \}$ and $Y=\{av, b^2u, avb^2u \}$.
 
If $x=a$, which is of order $3$, and $y=b^2u$ or $avb^2u$ then 
$\langle x \rangle \cap \langle y \rangle^g=1$ for all $g \in G_3$, as in the proof of Theorem \ref{ G_3 periodic}. When $x=b$ and $y=av$ or $avb^2u$, the same argument applies. If we are in one of the following cases: $x=a$ and $y=av$, or $x=b$ and $y=b^2u=(bu^2)^2$, then the  condition $\langle x \rangle \cap \langle y \rangle^g=1$ follows from Lemma \ref{intersection}.

It remains to check the case $x=ab$ and $y\in Y$. If $y=b^2u$, which is of order $3$, then we have  $\langle x \rangle \cap \langle y \rangle^g=1$, as in the previous paragraph. Now assume that $y=av$. Since $(av)^3=v^{a^2}v^av$, we have
\[
\psi_3((av)^3)=([a,b], [a,b], [a,b]).
\]
By Lemma \ref{center}, $(av)^3 \in Z(G_3)$. On the other hand, 
\[
\psi_3((ab)^3)=(b^a, b, b),
\]
and hence $(ab)^3 \notin Z(G_3)$. Thus, the condition 
$\langle x \rangle \cap \langle y \rangle^g=1$ follows.

Finally, we have to take $x=ab$ and $y=avb^2u$. Since $v \in \st_{G_3}(1)'$ and $\st_{G_3}(1)$ is of nilpotency class $2$, $v \in Z(\st_{G_3}(1))$. So, $y= ab^2vu$, and $y^3=(ab^2v)^3$. Observe that
\[
(ab^2v)^3=b_2^2b_1^2b_0^2v^{a^2}v^{a}v.
\]
By taking into account that $b_0b_1b_2=1$ in $G_2$, we get
\[
\psi_3(y^3)=(b^{-1},  (b^{-1})^a, (b^{-1})^a).
\]
If $\langle (ab)^3 \rangle = \langle y^3\rangle^g$ for some $g \in G_3$, then $(ab)^3=(y^{3i})^g$ for $i=1$ or $-1$. Since $\st_{G_2}(1)$ is abelian, the components of $\psi_3((y^{3i})^g)$ are of the form $(b^{-i})^{a^m}$, with $m \in \{0, 1, 2 \}$. Also one of the components of $\psi_3((ab)^3)$ is $b$. Then by equality $(ab)^3=(y^{3i})^g$, we have  $(b^{-i})^{a^m}=b$ in $G_2$ for some $m$. Thus, $[a^m, b^i]=b^{1+i} \in G_2'$, and this is true only if $1+i\equiv 0 \pmod{3}$. Thus 
\[
(ab)^3=(y^{-3})^g.
\] 
Note that
\begin{equation}
\label{last check}
\psi_3((ab)^3)=(b_1, b_0, b_0) \ \ \
\text{and} \ \ \
\psi_3(y^{-3})=(b_0, b_1, b_1).
\end{equation}
We write $g=a^i h$ for some $h \in \st_{G_3}(1)$ and $0 \leq i \leq 2$. Then the equality  $(ab)^3=(y^{-3})^g$ and (\ref{last check}) imply that $\psi_3(h)$ has to be congruent to one of the following modulo 
$\st_{G_2}(1) \times \st_{G_2}(1) \times \st_{G_2}(1)$: $(a, a^2, a^2)$, $(1, 1, a^2)$ or $(1, a^2, 1)$. However, since $G$ is periodic, the product of the powers of $a$ in the components of $\psi_3(h)$ has to be $1$. Thus, we conclude that there is no such $h$ in $G_3$, and therefore  $\langle (ab)^3 \rangle \neq \langle y^3\rangle^g$  for any $g\in G_3$.
This completes the proof.
\end{proof}

The following result, which gives a sufficient condition to lift a Beauville structure from a quotient group, is Lemma 4.2 in \cite{FJ}.

\begin{lem}
\label{lifting structure}
Let $G$ be a finite group and let $\{x_1,y_1\}$ and $\{x_2,y_2\}$ be two sets of generators of $G$.
Assume that, for a given $N\trianglelefteq G$, the following hold:
\begin{enumerate}
\item
$\{x_1N,y_1N\}$ and $\{x_2N,y_2N\}$ form a Beauville structure for $G/N$.
\item
$o(g)=o(gN)$ for every $g\in\{x_1,y_1,x_1y_1\}$.
\end{enumerate}
Then $\{x_1,y_1\}$ and $\{x_2,y_2\}$ form a Beauville structure for $G$.
\end{lem}

We are now ready to give the main result of this section.

\begin{thm}
\label{main thm: periodic}
Let $G$ be a periodic GGS-group over the $p$-adic tree. Then the quotient $G/\st_G(n)$ is a Beauville group if $p \geq5$ and $n\geq2$,  or $p=3$ and $n\geq 3$.
\end{thm}

\begin{proof}
By Theorem \ref{G_2 periodic}, $G/\st_G(2)$ is a Beauville group if and only if $p\geq 5$. Now we assume that $n\geq 3$. If $p\geq 5$ then by Theorem \ref{ G_3 periodic}, $G/\st_G(3)$ is a Beauville group with Beauville triples $X=\{ a^{-2}, ab, a^{-1}b \} $ and $Y=\{ab^2, b, a^3b \}$. Since $o(ab\st_G(n))= o(a^{-1}b\st_G(n))=p^2$ for any $n\geq 3$, we can apply Lemma \ref{lifting structure}, and hence $G/\st_G(n)$ is a Beauville group for every $n\geq 3$. Similarly, if $p=3$ then the Beauville structure of $G/\st_G(3)$ given in Theorem \ref{ G_3 periodic, p=3} is inherited by $G/\st_G(n)$  for any $n\geq 3$.
\end{proof}


\section{Quotients of non-periodic GGS-groups}

Let $G$ be a non-periodic GGS-group with defining vector $\bold{e}=(e_1, \dots, e_{p-1})$. In this section, we will prove that the quotients of $G$ by its level stabilizers $\st_G(n)$ are  not Beauville groups.

Since $G$ is non-periodic, we have $\Sigma_{i=1}^{p-1} e_i=\alpha\neq 0$. Thus,  by Lemma \ref{rank of C}, the rank of the circulant matrix $C(\bold{e}, 0)$ is $p$, and according to Proposition \ref{properties of G2},  $G/\st_G(2)$ is a $p$-group of maximal class of order $p^{p+1}$. Since 
\[
G/\st_G(2) \lesssim\Aut \TT /\st(2)\cong C_p \wr C_p,
\]
this implies that $G/\st_G(2) \cong C_p \wr C_p$, and thus it is of exponent $p^2$.

Note that the $p+1$ maximal subgroups of $G$ are $\langle a, G'\rangle$, $\langle b, G'  \rangle$ and $M_{i}=\langle ab^i, G'\rangle$ for all $1 \leq i \leq p-1$.
We write $M_{n,i}$ instead of $M_i/ \st_G(n)$, for every $1 \leq i \leq p-1$ and for $n\geq 2$. Then $M_{n,i}=\langle ab^i, G_n^{'}\rangle$ is a maximal subgroup of $G_n=G/\st_G(n)$.

The following proposition gives the relation between the $p$th powers of elements in $M_{n,i} \smallsetminus G_n^{'}$ for all $1\leq i \leq p-1$.

\begin{pro}
\label{collision}
Let $G_n=G/\st_G(n)$ for $n\geq 2$. Then the following hold:
\begin{enumerate}
\item
All elements in $M_{n,i} \smallsetminus G_n^{'}$ are of order $p^n$ for every $1\leq i \leq p-1$.
\item 
If $g$ is an element in $M_{n,i} \smallsetminus G_n^{'}$ such that $g=(ab^i)^kw$ for some $w\in G_n^{'}$ and for some $1 \leq k \leq p-1$, then 
\[
g^{p^{n-1}}=(ab^i)^{kp^{n-1}}.
\]
\item 
Cyclic subgroups generated  by $p^{n-1}$st powers of elements in $\bigcup_{i=1}^{p-1} M_{n,i} \smallsetminus G_n^{'}$ coincide.
\end{enumerate}
\end{pro}

\begin{proof}
	
We will first show the result for $n=2$. Since $G_2 \cong C_p \wr C_p$, all elements in $M_{2,i} \smallsetminus G_2^{'}$ are of order $p^2$ for $1\leq i \leq p-1$, and hence (i) holds. We know that for any element $u \in M_{2,i} \smallsetminus G_2^{'}$, we have $\Cl_{G_2}(u)=uG_2'$ and $u^p \in Z(G_2)$.  Thus, if $g=(ab^i)^kw$ for some $w\in G_2'$ then $g$ and $(ab^i)^k$ are conjugate in $G_2$. Since $(ab^i)^{kp} \in Z(G_2)$, this implies that in $G_2$
\[
g^p=(ab^i)^{kp},
\]
and so (ii) holds. It remains to show that (iii) holds. Since all elements in $\bigcup_{i=1}^{p-1} M_{2,i} \smallsetminus G_2^{'}$ are of order $p^2$,  cyclic subgroups generated by their $p$th powers are equal to $Z(G_2)$.

In order to prove the proposition we will use induction on $n$. Before proceeding to the induction step, consider the element $ab^i$ of $G$ for  $1\leq i \leq p-1$. We know that $(ab^i)^p=b_{p-1}^i b_{p-2}^i \dots b_1^ib_0^i$. Then the condition $\Sigma_{i=1}^{p-1} e_i=\alpha$ implies that
\[
\psi((ab^i)^p)=\big( a^{i\alpha}(b^i)^{a^{ie_1}}, a^{i\alpha}(b^i)^{a^{i(e_1+e_2)}}, \dots, b^ia^{i\alpha}, 
a^{i\alpha}b^i \big).
\]
Then 
\[
\psi((ab^i)^{kp})=\big( (a^{\alpha}b)^{ki} u_1, (a^{\alpha}b)^{ki} u_2, \dots, (a^{\alpha}b)^{ki} u_p\big),
\]
for $u_{\ell} \in G'$.
Let $g=(ab^i)^kw$ for some $w\in G'$. By the previous paragraph we know that
\[
g^p \equiv (ab^i)^{kp} \pmod{ \st_G(2)}.
\]
Write $g^p=(ab^i)^{kp}t$ for some $t \in \st_G(2)$. By Proposition 
\ref{comm of st_G(1)}, we have $\st_G(2)=[\st_G(1), \st_G(1)]$, and hence $\psi(t) \in G' \times  \overset{p}{\ldots} \times G'$. Thus
\[
\psi(g^p)=\big( (a^{\alpha}b)^{ki}w_1, (a^{\alpha}b)^{ki}w_2, \dots (a^{\alpha}b)^{ki}w_p \big),
\]
where $w_{\ell} \in G'$. 

Now assume that the proposition holds for $n\geq 2$. Then
\[
\psi_{n+1}((ab^i)^{kp^n})= \big( ((a^{\alpha}b)^{ki}u_1)^{p^{n-1}}, ((a^{\alpha}b)^{ki}u_2)^{p^{n-1}}, \dots ((a^{\alpha}b)^{ki}u_p)^{p^{n-1}} \big),
\]
\[
\psi_{n+1}(g^{p^n})= \big( ((a^{\alpha}b)^{ki}w_1)^{p^{n-1}}, ((a^{\alpha}b)^{ki}w_2)^{p^{n-1}}, \dots ((a^{\alpha}b)^{ki}w_p)^{p^{n-1}} \big),
\]
where each component is an element of $G_n$. By the induction hypothesis, all components are equal in $G_n$, and of order $p$. This completes the proof.
\end{proof}

\begin{thm}
\label{main thm non-periodic}
Let $G$ be a non-periodic GGS-group over the $p$-adic tree. Then the quotient $G/\st_G(n)$ is not a Beauville group for any $n \geq 1$.
\end{thm}

\begin{proof}
Let $G_n=G/\st_G(n)$. Clearly, $G_1$ is not a Beauville group since it is cyclic of order $p$. Thus, we assume that $n\geq2$. 

We argue by way of contradiction.  Suppose $\{x_1,y_1\}$ and $\{x_2,y_2\}$ are two systems of generators of $G_n$ such that $\Sigma(x_1,y_1)\cap \Sigma(x_2,y_2)=1$.
Since no two of the elements $x_1$, $y_1$ and $x_1y_1$ can lie in the same maximal subgroup of $G_n$, it follows from Proposition \ref{collision}(i) that one of these elements, say $x_1$, is of order $p^n$.
Similarly, we may assume that the order of $x_2$ is also $p^n$.
Then by Proposition \ref{collision}(iii), we conclude that $\langle x_1^{p^{n-1}} \rangle=\langle x_2^{p^{n-1}} \rangle$, which is a contradiction.
\end{proof}

\section*{Acknowledgments}
We would like to thank G. A. Fern\'andez-Alcober for proposing us this research problem and for helpful comments.

\end{document}